\documentclass[12pt,reqno]{amsart}
\usepackage{hyperref}
\usepackage{amsmath,amssymb}
\usepackage[english]{babel}
\usepackage{caption}
\usepackage{graphicx}
\usepackage{float}
\usepackage{pgf,tikz}
\usepackage{ upgreek }
\usepackage{mathrsfs}

\newcommand \reg{\operatorname{reg}}

\newcommand \cmdef{\operatorname{cmdef}}
\newcommand \Tor{\operatorname{Tor}}

\newcommand \depth{\operatorname{depth}}

\newcommand \pd{\operatorname{pd}}

\newcommand{\KK}{\mathbb{K}}
\newcommand{\sk}{{\ensuremath{\sf k }}}

\newtheorem{theorem}{Theorem}[section]

\newtheorem{lemma}[theorem]{Lemma}
\newtheorem{proposition}[theorem]{Proposition}
\newtheorem{example}[theorem]{Example}

\newtheorem{question}[theorem]{Question}
\newtheorem{remark}[theorem]{Remark}
\newtheorem{corollary}[theorem]{Corollary}

\begin{document}
\title{DEPTH AND EXTREMAL BETTI NUMBER OF BINOMIAL EDGE IDEALS}
\author{Arvind Kumar}
\email{arvkumar11@gmail.com}

\author{Rajib Sarkar}
\email{rajib.sarkar63@gmail.com}
\address{Department of Mathematics, Indian Institute of Technology
Madras, Chennai, INDIA - 600036}

\begin{abstract}
Let $G$ be a simple graph on the vertex set $[n]$ and $J_G$ be the corresponding binomial edge ideal.
Let $G=v*H$ be the cone of $v$ on $H$. In this article, we compute all the Betti numbers of $J_G$ in terms of Betti numbers of
$J_H$ and as a consequence, we get the Betti diagram of wheel graph. Also, we study Cohen-Macaulay defect of $S/J_G$ 
in terms of Cohen-Macaulay defect of $S_H/J_H$ and using this  we construct a graph with Cohen-Macaulay defect $q$ for any $q\geq 1$. We obtain the depth of binomial edge ideal of join of graphs. Also, we prove that for
any pair $(r,b)$ of positive integers with $1\leq b< r$, there exists a connected graph $G$ such that $\reg(S/J_G)=r$
and the number of extremal Betti numbers of $S/J_G$ is $b$.
\end{abstract}
\maketitle

\section{Introduction}
Let $R = \KK[x_1,\ldots,x_m]$ be the polynomial ring over an arbitrary
field $\KK$ and $M$ be a finitely generated graded  $R$-module. 
Let
\[
0 \longrightarrow \bigoplus_{j \in \mathbb{Z}} R(-p-j)^{\beta_{p,p+j}(M)} 
{\longrightarrow} \cdots {\longrightarrow} \bigoplus_{j \in \mathbb{Z}} R(-j)^{\beta_{0,j}(M)} 
{\longrightarrow} M\longrightarrow 0,
\]
be the minimal graded free resolution of $M$, where 
$R(-j)$ is the free $R$-module of rank $1$ generated in degree $j$. The number $\beta_{i,i+j}^R(M)$ is  the 
$(i,i+j)$-th graded Betti number of $M$. The
projective dimension  and Castelnuovo-Mumford regularity are two invariants associated with
$M$ that can be read off from the minimal free resolution. The 
Castelnuovo-Mumford regularity of $M$, denoted by $\reg(M)$, is defined as 
\[
\reg(M):=\max \{j : \beta_{i,i+j}^R(M) \neq 0\}
\] and the projective dimension of $M$, denoted by $\pd_R(M)$, is defined as \[\pd_R(M):=\max\{i : \beta_{i,i+j}^R(M) \neq 0\}.\]
A nonzero graded Betti number $\beta_{i,i+j}^R(M)$ is called an \textit{extremal Betti number}, if $\beta_{r,r+s}^R(M)=0$ 
for all pairs $(r,s)\neq (i,j)$ with $r\geq i$ and $s\geq j$. Observe that the extremal Betti number is unique if and only if 
$\beta_{p,p+r}^R(M)\neq 0$, where $p =\pd_R(M)$ and $r=\reg(M)$.

Let $G$ be a  simple graph on  $V(G)=\{1,2,\dots,n\}$
and edge set $E(G)$. Let $S=\KK[x_1,\dots,x_n,y_1,\dots,y_n]$ be the polynomial ring over an arbitrary
field $\KK$. The ideal $J_{G}$ generated by the binomials $x_iy_j-x_jy_i$, where $i<j$ and $\{i,j\}\in E(G)$, is known as the
binomial edge ideal of $G$. The notion of binomial edge ideal 
was introduced by Herzog et.al. in \cite{HH1} and independently by Ohtani in \cite{oh}. Algebraic properties and
invariants of binomial edge ideals have been studied by many authors, see \cite{EZ,ger1,KM1}. In particular, establishing a relationship between Castelnuovo-Mumford regularity (simply regularity), projective dimension, Hilbert series of binomial edge ideals and combinatorial invariants associated with graphs is an active area of 
research, see \cite{AN2017, JACM,KM4,FM,RKM}. In general, the
algebraic invariants such as regularity and depth of $J_{G}$ are hard
to compute. There are bounds known for the regularity and depth of
binomial edge ideals, see \cite{AN2017,MM}. The maximal possible depth
of binomial edge ideal of a connected graph on $n$ vertices is
$n+1$ (see \cite[Theorems 3.19, 3.20]{AN2017}). Also, if $G$ is a
connected graph on $n$ vertices such that $S/J_G$ is Cohen-Macaulay,
then $\depth_{S}(S/J_G)=n+1$. In \cite{Alba-MathNatch}, de Alba and Hoang
studied the depth of some subclass of closed graphs. However not much
more is known about the depth of binomial edge ideal. For an ideal $I
\subset S$, the Cohen-Macaulay defect of $S/I$ is defined to be
$\cmdef(S/I) := \dim(S/I) - \depth_S(S/I)$. We study the depth and
Cohen-Macaulay defect of $S/J_G$, where $G$ is a cone of $v$ on a
graph $H$, denoted by $v * H$ (for definition see section 3). We show that the depth remains invariant under the process
of taking cone on connected graph, (Theorem \ref{cone-depth}). As a
consequence, we prove that for any positive integer $q$, there exists
a graph having Cohen-Macaulay defect equal to $q$, (Corollary
\ref{cone-cmdeff}). We also compute the depth of $S/J_{v*H}$, when $H$
is a disconnected graph, (Theorem \ref{cone-depth1}).

Another homological invariant that helps in understanding more about
its structure is the Betti number. There have been few attempts in
computing the Betti numbers of binomial edge ideals, for example,
Zafar and Zahid for cycles, \cite{SZ}, Schenzel and Zafar for complete
bipartite graphs, \cite{Schenzel}, Jayanthan et al.  for trees and
unicyclic graphs \cite{JAR}. Extremal Betti numbers of binomial edge
ideals of closed graphs were studied by de Alba and Hoang in
\cite{Alba-MathNatch}. In \cite{HR-Extremal}, Herzog and Rinaldo
studied extremal Betti number of binomial edge ideal of block graphs.
We compute all the Betti numbers of cone of a graph, (Theorem
\ref{Betti-cone}). As a consequence, we obtain the  Betti numbers of
binomial edge ideal of wheel graph, (Corollary \ref{Betti-wheel}).

We then consider a more general form of cone, namely the join
product of two arbitrary graphs. Given two graphs
$G_1$ and $G_2$, it is interesting to understand the properties of
$G_1*G_2$ (for definition see section 4) in terms of the corresponding
properties of $G_1$ and $G_2$. In \cite{KMJA}, Kiani and Saeedi Madani
studied the regularity of  $J_{G_1*G_2}$.  We computed the Hilbert
series of binomial edge ideal of $G_1*G_2$ in terms of the Hilbert
series of $J_{G_1}$ and $J_{G_2}$, \cite{AR1}. In this article, we
study the depth of  $S/J_{G_1*G_2}$, (Theorems \ref{depththm2},
\ref{depththm1}, \ref{depththm3}).  As a consequence, we obtain the
depth of complete multipartite graphs, (Corollary
\ref{com-bipartite}).

Recently, researchers are trying to construct graphs such that their
corresponding edge ideals satisfy certain algebraic properties.  For a
given pair of positive integers $(r,s)$, Hibi and Matsuda in
\cite{HMmono} showed the existence of monomial ideal $I_{r,s}$ such
that $\reg(S/I_{r,s})=r$ and the degree of $h$-polynomial of
$S/I_{r,s}$ is $s$. In \cite{HM-Extremal}, Hibi et al. constructed a
graph $G$ such that for $1 \leq b \leq r$, the regularity of the
monomial edge ideal of $G$ is $r$ and the number of its extremal Betti
numbers is $b$. Given a pair $(r,s)$ with $1\leq r\leq s$, Hibi and
Matsuda constructed a graph $G$  such that $\reg(S/J_G)=r$ and the
degree of $h$-polynomial of $S/J_G$ is  $s$, \cite{HibiM2018}.
In this article, we construct  a graph $G$ such that $\reg(S/J_G)=r$
and the number of extremal Betti numbers of $S/J_G$ is $b$, for $1\leq
b< r$ (Theorem \ref{exist}). 
\section{Preliminaries}
In this section, we  recall  some notation and fundamental results on
graphs and their corresponding binomial edge ideals.

Let $G$  be a  finite simple graph with vertex set $V(G)$ and edge set
$E(G)$. For  $A \subseteq V(G)$, $G[A]$ denotes the induced
subgraph of $G$ on the vertex set $A$, that is, for $i, j \in A$, $\{i,j\} \in E(G[A])$ if and only if $ \{i,j\} \in E(G)$. 
For a vertex $v$, $G \setminus v$ denotes the  induced subgraph of $G$
on the vertex set $V(G) \setminus \{v\}$. A vertex $v \in V(G)$ is
said to be a \textit{cut vertex} if $G \setminus v$ has  more  components than $G$.  We say that $G$ is $k$ \textit{vertex-connected} if $k <n$ and for every 
$A \subset [n]$ with $|A| <k$, the induced graph $G[ \bar{A}]$ is connected, where $\bar{A}=[n]\setminus A$. The \textit{vertex connectivity} of a connected  graph $G$, denoted by $\kappa(G)$, is defined as the maximum positive integer $k$ such that $G$ is  $k$ vertex-connected. 
A subset $U$ of $V(G)$ is said to be a 
\textit{clique} if $G[U]$ is a complete graph. We denote the number of cliques of cardinality $i$ in $G$ by $\sk_i(G)$. A vertex $v$ is said to be a \textit{simplicial vertex} if it belongs to exactly one maximal clique.  For a vertex $v$, $N_G(v) = \{u \in V(G) ~
: ~ \{u,v\} \in E(G)\}$ denotes neighborhood of $v$ and  $G_v$ is the
graph on the vertex set $V(G)$ and edge set $E(G_v) =E(G) \cup \{
\{u,w\}: u,w \in N_G(v)\}$. A component  of $G$ is said to be a \textit{nontrivial} component if it has atleast one edge.

For  $T \subset [n]$, let $\bar{T} = [n]\setminus T$ and $c_T$
denote the number of connected components of $G[\bar{T}]$. Let $G_1,\cdots,G_{c_T}$ be the connected 
components of $G[\bar{T}]$. For each $i$, let $\tilde{G_i}$ denote the complete graph on $V(G_i)$ and
$P_T(G) = (\underset{i\in T} \cup \{x_i,y_i\}, J_{\tilde{G_1}},\cdots, J_{\tilde{G}_{c_T}}).$
It was shown by Herzog et al.  that $J_G =  \underset{T \subseteq [n]}\cap P_T(G)$, \cite[Theorem 3.2]{HH1}.
For each $i \in T$, if $i$ is a cut vertex of the graph $G[\bar{T} \cup \{i\}]$,
then we say that $T$ has the cut point property. Set $\mathcal{C}(G) =\{\emptyset \}
\cup \{ T: T \; \text{has the cut point property} \}$. It follows from  \cite[Corollary 3.9]{HH1} 
that $T \in \mathcal{C}(G)$ if and only if $P_T(G)$ is a minimal prime of $J_G$. It follows from the Auslander-Buchsbaum formula that $\depth_S(S/J_G)=2n-\pd_S(S/J_G)$.

The following basic property of depth is used repeatedly in this
article.

\begin{lemma}\label{depth-lemma}
	Let $S$ be a standard graded polynomial ring  and $M,N$ and $P$ be finitely generated graded $S$-modules. 
	If $ 0 \rightarrow M \xrightarrow{f}  N \xrightarrow{g} P \rightarrow 0$ is a 
	short exact sequence with $f,g$  
	graded homomorphisms of degree zero, then 
	\begin{enumerate}
		\item[(i)] $\depth_S(M) \ge \min \{\depth_S(N), \depth_S(P)+1\},$
		\item[(ii)] $\depth_S(M) = \depth_S(P)+1$ if  $\depth_S(N) > \depth_S(P)$,
		\item[(iii)] $\depth_S(M) = \depth_S(N)$ if  $\depth_S(N) < \depth_S(P)$.
	\end{enumerate}	
\end{lemma}

\section{Binomial edge ideal of cone of a graph}
In this section, we study the binomial edge ideal of cone of a graph.  
Let $H$ be a graph on the vertex set $[n]$. The cone of $v$ on $H$, denoted by $v*H$, is the  graph with
the vertex set $V(v*H)=V(H)\sqcup \{v\}$ and edge set $E(v*H)=E(H)\sqcup \{ \{v,u\} \mid u\in V(H)\}.$ From now, we assume
that $H$ is not a complete graph. Set $G=v*H$,  $S_H=\mathbb{K}[x_i,y_i:i \in V(H)]$ and $S=S_H[x_v,y_v]$. First, we 
recall a lemma due to Ohtani which is useful in this section.

\begin{lemma}(\cite[Lemma 4.8]{oh})\label{ohtani-lemma}
	Let $G$ be a  graph on $V(G)$ and $v\in V(G)$ such that $v$ is not a simplicial vertex. Then $J_{G}=(J_{G\setminus v}+(x_v,y_v))\cap J_{G_v}$.
\end{lemma}
One can see that if $G=v*H$, then $ {G_v}=K_{n+1}$, $G_v \setminus v= K_n$ and $G \setminus v=H$. Therefore, $
(x_v,y_v)+J_{G\setminus v}+J_{G_v}=(x_v,y_v)+J_{K_n}$. Thus, by Lemma  \ref{ohtani-lemma}, we have the following short exact sequence:
\begin{align}\label{cone-ses}
	0\longrightarrow \frac{S}{J_{G}}\longrightarrow 
	\frac{S}{(x_v,y_v)+J_H} \oplus \frac{S}{J_{K_{n+1}}}\longrightarrow \frac{S}{(x_v,y_v)+J_{K_{n}}} \longrightarrow 0.
\end{align}

\begin{remark}\label{pd-lower}
	It follows from \cite[Theorem 1.1]{her1} that if $G$ is a complete graph on $[n]$, then $S/J_G$ is Cohen-Macaulay of dimension $n+1$. 
	If $G$ is a  connected graph which is not a complete graph, then $\kappa(G)\geq1$. Therefore, by \cite[Theorems 3.19, 3.20]{AN2017}, we get  
	$\pd_S(S/J_G)\geq n-2+\kappa(G)\geq n-1$. Thus, for any connected graph $G$,  $\pd_S(S/J_G)\geq n-1$ and hence, by Auslander-Buchsbaum formula, $\depth_{S}(S/J_{G}) \leq n+1$.
\end{remark}

We proceed to prove the following lemma which plays an important role. 
\begin{lemma}\label{linearstand-binomial}
	Let $G$ be a connected graph on the vertex set $[n]$. Let $p=\pd_S(S/J_{G})$. Then  $\beta_{p,p+1}^S(S/J_G)\neq 0$ if and 
	only if $G$ is a complete graph. Moreover, if $\beta_{i,i+2}^S(S/J_G) $ is an extremal Betti number, then $i=p$.
\end{lemma}
\begin{proof}
	By Remark \ref{pd-lower}, $p \geq n-1$.
	It follows from \cite[Corollary 4.3]{HKS} that $\beta_{p,p+1}^S(S/J_{G})=p \sk_{p+1}(G)$. Therefore, $\beta_{p,p+1}^S(S/J_G)\neq 0$
	if and only if $G$ is a complete graph. Now, if possible assume that $i<p$. Since $\beta_{i,i+2}^S(S/J_G) $ is an extremal Betti
	number, $\beta_{p,p+j}^S(S/J_G)=0$ for $j\geq 2$, which implies that $\beta_{p,p+1}^S(S/J_G)$ must be an extremal Betti number. Thus,
	$G$ is a complete graph which contradicts \cite[Theorem 2.1]{KM1}, as $\reg(S/J_G) \geq 2$. Hence, the assertion follows.
\end{proof}
Let $M$ be a finite graded $S$-module. The \textit{Cohen-Macaulay defect}, denoted by $\cmdef(M)$, is defined by $\dim(M)-\depth_{S}(M)$. 
A graded $S$-module $M$ is said to be \textit{almost Cohen-Macaulay} if
$\cmdef(M)=1$. A graph $G$ is said to be \textit{(almost) Cohen-Macaulay} if $S/J_G$ is 
(almost) Cohen-Macaulay. 

First, we recall some basic facts about Betti numbers and minimal free resolution. Let $R=\KK[x_1,\dots,x_m]$, $R'=\KK[x_{m+1},\dots,x_n]$ and $T=\KK[x_1,\dots,x_n]$ be
polynomial rings. Let $I\subseteq R$ and $J\subseteq R'$ be homogeneous ideals. Then minimal free resolution of $T/(I+J)$ is tensor product of minimal free resolutions of $R/I$ and $R'/J$. Also, for all $i,j$, 
\begin{align}\label{Bettiproduct}
	\beta_{i,i+j}^T\left(\frac{T}{I+J}\right)= \underset{{\substack{i_1+i_2=i \\ j_1+j_2=j}}}{\sum}\beta_{i_1,i_1+j_1}^R\left(\frac{R}{I}\right)\beta_{i_2,i_2+j_2}^{R'}\left(\frac{R'}{J}\right).
\end{align}
Now, we construct almost Cohen-Macaulay graphs.
\begin{theorem}\label{cone-depth}
	Let $H$ be a connected graph on the vertex set $[n]$ and $G=v*H$ be the cone of $v$ on  $H$. Then, $\depth_S(S/J_G) =\depth_{S_H}(S_H/J_H).$
	In particular, if $H$ is  Cohen-Macaulay, then $G$ is  almost Cohen-Macaulay.
\end{theorem}
\begin{proof}
	Assume that $\depth_{S_H}(S_H/J_H)=n+1$. Therefore, $\pd_{S_H}(S_H/J_H)=n-1$ and $\pd_S(S/((x_v,y_v)+J_H))=n+1$. Also, we have 
	$\pd_{S}(S/((x_v,y_v)+J_{K_n}))=n+1$ and  $\pd_{S}(S/J_{K_{n+1}})=n$. Since, $H$ is a connected graph, by Lemma \ref{linearstand-binomial}, there exists a $j \geq 2$ such that
	$\beta_{n-1,n-1+j}^{S_H}(S_H/J_H) \neq 0$. Consider, the long exact sequence of Tor corresponding to \eqref{cone-ses},
	\begin{align*}
		0 \rightarrow \Tor_{n+1,n+1+j}^{S}\left( \frac{S}{J_G},\KK\right)\rightarrow \Tor_{n+1,n+1+j}^{S}\left( \frac{S}{(x_v,y_v)+J_{H}},\KK\right) \rightarrow 0.
	\end{align*}
	Since $\beta_{n+1,n+1+j}^S(S/((x_v,y_v)+J_H)) \neq 0$, $\beta_{n+1,n+1+j}^S(S/J_G)\neq 0$. Therefore, $\pd_S(S/J_G)\geq n+1$ and 
	hence by Auslander-Buchsbaum formula, $\depth_S(S/J_G) \leq n+1$. Now
	using Lemma \ref{depth-lemma} on the short exact sequence \eqref{cone-ses}, we get that
	\[\depth_S(S/J_G)\geq  \min \{\depth_{S_H}(S_H/J_H),n+2\}=n+1.\]
	Hence, $\depth_{S}(S/J_G)=n+1$.  If $\depth_{S_H}(S_H/J_H)<n+1$, then by Lemma \ref{depth-lemma}, $\depth_S(S/J_G)=\depth_{S_H}(S_H/J_H).$ Now, if $H$ is
	Cohen-Macaulay, then $\depth_{S}(S/J_G)=n+1$. It follows from \cite[Theorem 4.6]{AR1} that $\dim(S/J_G)=n+2$. Hence, $G$ is an almost Cohen-Macaulay. 
\end{proof}
\begin{theorem}\label{cone-cmdef}
	Let $H$ be a connected graph on the vertex set $[n]$ and $G=v*H$ be the cone of $v$ on $H$. If $\dim(S_H/J_H) \geq n+2$,
	then $\cmdef(S/J_G)=\cmdef(S_H/J_H)$ and otherwise $\cmdef(S/J_G)=\cmdef(S_H/J_H)+1.$
\end{theorem}
\begin{proof}
	It follows from \cite[Theorem 4.6]{AR1} that if $\dim(S_H/J_H) \geq n+2$, then $\dim(S/J_G)=\dim(S_H/J_H)$. Thus, by  
	Theorem \ref{cone-depth}, $\cmdef(S/J_G)=\cmdef(S_H/J_H)$. Now, if $\dim(S_H/J_H)=n+1$, then again by  \cite[Theorem 4.6]{AR1},  $\dim(S/J_G)=n+2$ and hence $\cmdef(S/J_G)=\cmdef(S_H/J_H)+1.$	
\end{proof}
We now show that one can construct graphs with as large Cohen-Macaulay
defect as one wants.
\begin{corollary}\label{cone-cmdeff}
	Let $H$ be a connected graph on $[n]$ and $q$ be a positive
	integer. If $G= K_q*H$, then $\depth_{S}(S/J_G)=\depth_{S_H}(S_H/J_H)$.
	In particular, if $H$ is  Cohen-Macaulay, then $\cmdef(S/J_G)=q$.
\end{corollary}
\begin{proof}
	Let $v_1,\ldots, v_q$ be vertices of $K_q$. Observe that $K_q*H =v_1*(\cdots * (v_q * H)\cdots)$. By  recursively applying 
	Theorem \ref{cone-depth}, $\depth_S(S/J_G)=\depth_{S_H}(S_H/J_H)$. Now, if $H$ is Cohen-Macaulay, then $\depth_S(S/J_G)=n+1$ and 
	it follows from \cite[Theorem 4.12]{AR1} that $\dim(S/J_G)=n+q+1$. Hence, the assertion follows.
\end{proof}
Let $G=K_q*H$, then by \cite[Theorem 4.12]{AR1} and Corollary \ref{cone-cmdeff}, if $\dim(S_H/J_H) \geq n+q+1$, 
then $\cmdef(S/J_G)=\cmdef(S_H/J_H)$ otherwise $\cmdef(S/J_G)=n+q+1-\dim(S_H/J_H)+\cmdef(S_H/J_H)$. 

To compute the depth formula for cone of a disconnected graph, we need the following lemma. 
\begin{lemma}\label{linearstand-binomial2}
	Let $G$ be a disconnected graph on the vertex set $[n]$. Assume that $G$ has atleast two nontrivial components. 
	Let $p=\pd_S(S/J_{G})$. Then  $\beta_{p,p+1}^S(S/J_G)=0$. Moreover, if $\beta_{i,i+2}^S(S/J_G) $ is an extremal Betti number, then $i=p$.
\end{lemma}
\begin{proof}
	Let $H_1,\ldots,H_q$ be nontrivial connected components of $G$ with $q \geq 2$. By Remark \ref{pd-lower}, for each $i \in [q]$,
	$\pd_{S_{H_i}}(S_{H_i}/J_{H_i})\geq |V(H_i)|-1$, where $S_{H_i}=\KK[x_v,y_v: v \in V(H_i)]$. Let $m= \sum_{i=1}^q |V(H_i)|$. Thus,   
	$p\geq m-q.$  It follows from \cite[Corollary 4.3]{HKS}, that $\beta_{p,p+1}^S(S/J_{G})=p \sk_{p+1}(G)$. If possible, $\beta_{p,p+1}^S(S/J_G)\neq 0$, 
	then $G$ has an induced clique of size atleast $m-q+1$, which is a contradiction. Now, if possible assume that $i<p$ which implies that $\beta_{p,p+1}^S(S/J_G)$
	is an extremal Betti number, which is a contradiction as $\beta_{p,p+1}^S(S/J_G)=0$. Hence, the assertion follows.
\end{proof} 
\begin{remark}\label{linearstand-binomial3}
	Let $G$ be a disconnected graph on $[n]$. If $\depth_S(S/J_G)=n+1$, then either $G$ has atleast two nontrivial components or  $G$ 
	has exactly one nontrivial component which is not a complete graph. Moreover, $\beta_{n-1,n}^S(S/J_G)=0$.
\end{remark}
We now compute the depth formula for cone of a disconnected graph.
\begin{theorem}\label{cone-depth1}
	Let	$G=v* H$, where $H$ is a disconnected graph on $[n]$. Then  \[\depth_S(S/J_G)=\min \{\depth_{S_H}(S_H/J_H),n+2\}.\]
\end{theorem}
\begin{proof}
	If $\depth_{S_H}(S/J_H)<n+1$, then by using Lemma \ref{depth-lemma} in the short exact sequence \eqref{cone-ses}, we have $\depth_S(S/J_G)=\depth_{S_H}(S_H/J_H).$ 
	Also, if  $\depth_{S_H}(S_H/J_H)>n+1$, then by virtue of Lemma \ref{depth-lemma}  $\depth_S(S/J_G)=n+2$. Now, assume that $\depth_{S_H}(S_H/J_H)=n+1$.   
	Observe that $\pd_{S_H}(S_H/J_H)=n-1$, $\pd_{S}(S/((x_v,y_v)+J_H)) =n+1$ and $\pd_{S}(S/((x_v,y_v)+J_{K_n})) =n+1$. By Remark \ref{linearstand-binomial3},
	there exists $j \geq 2$ such that $\beta_{n-1,n-1+j}^{S_H}(S_H/J_H) \neq 0$. Now consider, the long exact sequence of Tor corresponding to \eqref{cone-ses},
	\begin{align*}
		0 \rightarrow \Tor_{n+1,n+1+j}^{S}\left( \frac{S}{J_G},\KK\right)\rightarrow \Tor_{n+1,n+1+j}^{S}\left( \frac{S}{(x_v,y_v)+J_{H}},\KK\right) \rightarrow 0.
	\end{align*}
	Since, $\beta_{n+1,n+1+j}^S(S/((x_v,y_v)+J_H)) \neq 0$ and hence $\beta_{n+1,n+1+j}^S(S/J_G)\neq 0$. Therefore, $\pd_S(S/J_G)\geq n+1$ and hence 
	$\depth_S(S/J_G) \leq n+1$. Using Lemma \ref{depth-lemma}, we have $\depth_S(S/J_G)\geq n+1$ and this completes the proof.
\end{proof}
Also, if $G=K_q *H$, where $H$ is a disconnected graph, then by Theorems \ref{cone-depth}, \ref{cone-depth1}, $\depth_S(S/J_G)=\min\{\depth_{S_H}(S_H/J_H),n+2\}$. 
Now we compute the Betti numbers of $S/J_{v*H}$ in terms of the Betti numbers of $S_H/J_H$.

\begin{theorem}\label{Betti-cone}
	Let $H$ be a graph on the vertex set $[n]$. Let $G=v*H$ be the cone of $v$ on $H$.  Then, for  $i,j$,  
	\[\beta_{i,i+j}^{S}\left(\frac{S}{J_{G}}\right)= \left\{
	\begin{array}{ll}
	i\left(\sk_i(H)+\sk_{i+1}(H)\right), & \text{ if } j=1 \\
	\beta_{i,i+2}^{S_H}\left(\frac{S_H}{J_{H}}\right)+2\beta_{i-1,i+1}^{S_H}\left(\frac{S_H}{J_{H}}\right)+\beta_{i-2,i}^{S_H}\left(\frac{S_H}{J_{H}}\right)\\ +(i-1)\binom{n+1}{i+1}- (i-1)\sk_{i}(H)-(i-1)\sk_{i+1}(H), & \text{ if } j=2 \\
	\beta_{i,i+j}^{S_H}\left(\frac{S_H}{J_{H}}\right)+2\beta_{i-1,i-1+j}^{S_H}\left(\frac{S_H}{J_{H}}\right)+\beta_{i-2,i-2+j}^{S_H}\left(\frac{S_H}{J_{H}}\right), & \text{ if } j\geq 3,
	\end{array} \right.\]
	where $\beta_{i-2,i-2+j}^S\left(\frac{S}{J_{G}}\right)=0$ and $\beta_{i-1,i-1+j}^S\left(\frac{S}{J_{G}}\right)=0$, if $i-2<0$ and $i-1<0$ respectively.
\end{theorem}
\begin{proof}
	It follows from \cite[Corollary 4.3]{HKS} that $\beta_{i,i+1}^{S}(S/J_G)=i\sk_{i+1}(G).$
	Let $U$ be a clique in $G$ on $(i+1)$-vertices. Then either $v\in U$ or $v\notin U$. If $v\notin U$, then $U$ is a clique in $H$ on $(i+1)$-vertices. 
	If $v\in U$, then $U\setminus \{v\}$ is a clique in $H$ on $i$-vertices. Therefore, $\sk_{i+1}(G)=\sk_i(H)+\sk_{i+1}(H)$ and hence 
	$\beta_{i,i+1}^{S}(S/J_G)=i\left(\sk_i(H)+\sk_{i+1}(H)\right)$.
	Now, consider the  long exact sequence of Tor modules corresponding to the short exact sequence \eqref{cone-ses}:
	\begin{multline*}
		\cdots \rightarrow \Tor_{i,i+j}^{S}\left( \frac{S}{J_G},\KK\right)\rightarrow \Tor_{i,i+j}^{S}\left( \frac{S}{(x_v,y_v)+J_{H}},\KK\right) 
		\oplus \Tor_{i,i+j}^{S}\left(\frac{S}{J_{K_{n+1}}},\KK\right) \\ \rightarrow \Tor_{i,i+j}^{S}\left(\frac{S}{(x_v,y_v)+J_{K_{n}}},\KK\right) \rightarrow \Tor_{i-1,i+j}^{S}\left( \frac{S}{J_G},\KK\right)\rightarrow \cdots
	\end{multline*}
	For $j=2$, the above long exact sequence of Tor gives us
	\begin{align*}
		\beta_{i,i+2}^{S}\left(\frac{S}{J_G}\right)& =\beta_{i,i+2}^{S}\left(\frac{S}{(x_v,y_v)+J_H}\right)+\beta_{i+1,i+2}^{S}\left(\frac{S}{(x_v,y_v)+J_{K_{n}}}\right)-
		\beta_{i+1,i+2}^{S}\left(\frac{S}{J_{K_{n+1}}}\right)\\ &-\beta_{i+1,i+2}^{S}\left(\frac{S}{(x_v,y_v)+J_H}\right)  +\beta_{i+1,i+2}^{S}\left(\frac{S}{J_G}\right).
	\end{align*}
	By \eqref{Bettiproduct}, we have 
	\begin{align*}
		\beta_{i+1,i+2}^S\left(\frac{S}{((x_v,y_v)+J_{K_n})}\right)&=\beta_{i+1,i+2}^{S_H}\left(\frac{S_H}{J_{K_n}}\right)+2\beta_{i,i+1}^{S_H}\left(\frac{S_H}{J_{K_n}}\right)+\beta_{i-1,i}^{S_H}\left(\frac{S_H}{J_{K_n}}\right)\\
		&=(i+1)\binom{n}{i+2}+2i\binom{n}{i+1}+(i-1)\binom{n}{i}
	\end{align*} and
	\begin{align*}
		\beta_{i,i+2}^{S}\left(\frac{S}{(x_v,y_v)+J_H}\right) = \beta_{i,i+2}^{S_H}\left(\frac{S_H}{J_H}\right)+2\beta_{i-1,i+1}^{S_H}\left(\frac{S_H}{J_H}\right)+\beta_{i-2,i}^{S_H}\left(\frac{S_H}{J_H}\right).
	\end{align*}
	Therefore, we have \begin{align*}
		\beta_{i,i+2}^{S}\left(\frac{S}{J_G}\right)& = \beta_{i,i+2}^{S_H}\left(\frac{S_H}{J_H}\right)+2\beta_{i-1,i+1}^{S_H}\left(\frac{S_H}{J_H}\right)+\beta_{i-2,i}^{S_H}\left(\frac{S_H}{J_H}\right)\\&+(i-1)\binom{n+1}{i+1}
		-(i-1)\sk_i(H)-(i-1)\sk_{i+1}(H).
	\end{align*}
	
	Now let $j\geq 3$. Since, $\reg(S/((x_v,y_v)+J_{K_{n}}))=\reg(S/J_{K_{n+1}})=1$, \[\Tor_{i,i+j}^{S}\left(\frac{S}{J_{K_{n+1}}},\KK\right)=
	\Tor_{i,i+j}^{S}\left(\frac{S}{(x_v,y_v)+J_{K_{n}}},\KK\right)= \Tor_{i+1,i+j}^{S}\left(\frac{S}{(x_v,y_v)+J_{K_{n}}},\KK\right)=0.\]
	Then for $j\geq 3$, $\Tor_{i,i+j}^{S}\left( \frac{S}{J_{G}},\KK\right)\simeq \Tor_{i,i+j}^{S}\left( \frac{S}{(x_v,y_v)+J_{H}},\KK\right)$ and hence by virtue of \eqref{Bettiproduct}, we have 
	\[\beta_{i,i+j}^S\left(\frac{S}{J_G}\right)=\beta_{i,i+j}^{S_H}\left(\frac{S_H}{J_{H}}\right)+2\beta_{i-1,i-1+j}^{S_H}\left(\frac{S_H}{J_{H}}\right)+\beta_{i-2,i-2+j}^{S_H}\left(\frac{S_H}{J_{H}}\right),\] which proves our result.
\end{proof}
Let $G=K_q*H$ be the join of a complete graph and $H$. Then by using the above theorem recursively, one can compute all the Betti numbers of $S/J_G$.
Now, we compute the Betti diagram of the wheel graph. The wheel graph, denoted by $W_n$, is the cone of $v$ on $C_n$, $n \geq 4$.
\begin{corollary}\label{Betti-wheel}
	Let $W_n=v*C_n$ be the wheel graph with $n\geq 4$. Then $\reg(S/J_{W_n})=n-2$, $\pd_S(S/J_{W_n})=n+2$ and the Betti diagram of $S/J_{W_n}$ looks like the following:
	\[
	\begin{array}{c|ccccccccccc}
	
	& 0 & 1  & 2 & \cdots & i & i+1  &  i+2  & \cdots & n & n+1 & n+2 \\
	\hline
	0 & 1 & 0 & 0 & \cdots & 0 & 0 & 0 & \cdots & 0 & 0 & 0 \\
	1 & 0 & \beta_{1,2} & \beta_{2,3} & \cdots & 0  & 0 & 0 & \cdots & 0 & 0 & 0  \\
	2 & 0 & 0 & \beta_{2,4} & \cdots & \beta_{i,i+2} & \beta_{i+1,i+3} & \beta_{i+2,i+4} & \cdots & \beta_{n,n+2} & 0  & 0 \\
	\vdots & \vdots & \vdots & \vdots & \vdots  & \vdots & \vdots & \ddots & \ddots
	& \vdots & \vdots\\
	i & 0 & 0 & 0 & \cdots & \beta_{i,2i} & \beta_{i+1,2i+1} & \beta_{i+2,2i+2} & \ddots & \ddots & 0 & 0  \\
	\vdots & \vdots & \vdots & \vdots & \vdots  & \vdots & \vdots & \ddots & \ddots
	& \vdots & \vdots\\
	n-2 & 0 & 0 &\beta_{2,n}   & \cdots & \beta_{i,n-2+i} & \beta_{i+1,n-1+i} & \beta_{i+2,n+i} & \cdots & \beta_{n,2n-2} & \beta_{n+1,2n-1} & \beta_{n+2,2n}
	
	\end{array}
	\]
	where, $\beta_{1,2}=2n$, $\beta_{2,3}=2n$, $\beta_{2,4}=\binom{n} {2}+\binom{n+1} {3}-n$, $\beta_{3,5}=2 \binom{n} {2}+2 \binom{n+1} {4}$, $\beta_{4,6}= \binom{n} {2}+3 \binom{n+1} {5}$,
	
	\begin{center}
		$\beta_{i,i+2}=(i-1)\binom{n+1}{i+1}\text{, if  } i=5,\ldots,n,$
	\end{center}
	
	\begin{center}
		$\beta_{i,2i}=\beta_{i+2,2i+2}=\binom{n}{i},\beta_{i+1,2i+1}=2\binom{n}{i} \text{, if  } i=3,\ldots,n-3,$
	\end{center}
	
	\begin{center}
		$\beta_{i,i+n-2}=(n+1-i)\binom{n+2}{i-2}+2\binom{n+1}{i-3}\text{, if  } i=2,\ldots,n-3,$
	\end{center}	
	$\beta_{n-2,2n-4}= \binom{n} {2} +3 \binom{n+2} {6}+2 
	\binom{n+1} {6}$, $\beta_{n-1,2n-3}=2 \binom{n+1} {3} +2 \binom{n} {4}+4 \binom{n+1} {5}$, $\beta_{n,2n-2}=\binom{n-1} {2}-1 + \binom{n+2} {4}+ 2
	\binom{n+1} {4}$, $\beta_{n+1,2n-1}=2 \binom{n-1} {2}-2 + 2 \binom{n} {3}$ and $\beta_{n+2,2n}=\binom{n-1} {2}-1 $.
\end{corollary}
\begin{proof}
	The assertion follows from \cite[Corollary 16]{SZ} and Theorem \ref{Betti-cone}.
\end{proof}
Now, we study the position of extremal Betti number of $S/J_G$ in terms of the position of extremal Betti number of $S_H/J_H$. 
\begin{proposition}\label{extremal-cone}
	Let $H$ be a connected graph on the vertex set $[n]$. Let $G=v*H$ be the cone of $v$ on $H$. If $\beta_{i,i+j}^{S_H}(S_H/J_H)$ is an 
	extremal Betti number, then $\beta_{i+2,i+2+j}^{S}(S/J_G)$ is an extremal Betti number and both are equal. Moreover, if $\beta_{k,k+l}^S(S/J_G)$ is 
	an extremal Betti number, then $\beta_{k-2,k+l-2}^{S_H}(S_H/J_H)$ is an extremal Betti number and $\beta_{k,k+l}^S(S/J_G)=\beta_{k-2,k+l-2}^{S_H}(S_H/J_H)$.
\end{proposition}
\begin{proof}
	Let $\beta_{i,i+j}^{S_H}(S_H/J_H)$ be an extremal Betti number of $S_H/J_H$. Since $H$ is not a complete graph, by Lemma \ref{linearstand-binomial}, 
	$j\geq 2$. If $j\geq 3$, then by Theorem \ref{Betti-cone}, $\beta_{i+2,i+2+j}^{S}(S/J_G)=\beta_{i,i+j}^{S_H}(S_H/J_H)$ and for any pair $(r,s)$ with $r\geq i+2$, 
	$s\geq j$ and $(r,s)\neq (i+2,j)$,  $\beta_{r,r+s}^S(S/J_G)=0$. Let $p=\pd_{S_H}(S_H/J_H)$. If $j=2$, then by Lemma \ref{linearstand-binomial}, $\beta_{p,p+2}^{S_H}(S_H/J_H)$ 
	is an extremal Betti number. Therefore it follows from Theorem \ref{Betti-cone} that
	$$ \beta_{p+2,p+4}^{S}(S/J_G)= \beta_{p,p+2}^{S_H}(S_H/J_H) +(p+1) \binom{n+1} {p+3} -(p+1)\sk_{p+2}(H)-(p+1)\sk_{p+3}(H).$$
	By Remark \ref{pd-lower}, $p\geq n-1$, therefore, $\beta_{p+2,p+4}^S(S/J_G)=\beta_{p,p+2}^{S_H}(S_H/J_H)$. Now, let $\beta_{k,k+l}^S(S/J_G)$ is an extremal Betti 
	number. Therefore, by Lemma \ref{linearstand-binomial}, $l \geq 2$. Assume that $l\geq 3$. If possible, $\beta_{k-2,k-2+l}^{S_H}(S_H/J_H)$ is not an extremal Betti number. 
	Thus, there exists $r\geq k-2$ and $s \geq l$ such that $(r,s)\neq (k-2,l)$ and $\beta_{r,r+s}^{S_H}(S_H/J_H)\neq 0$. Therefore, by virtue of Theorem \ref{Betti-cone},
	$\beta_{r+2,r+2+s}^S(S/J_G)\neq 0$ which is a contradiction. Hence, $\beta_{k-2,k-2+l}^{S_H}(S_H/J_H)$ is an extremal Betti number and by Theorem \ref{Betti-cone},
	$\beta_{k,k+l}^S(S/J_G)=\beta_{k-2,k-2+l}^{S_H}(S_H/J_H)$. Now, if $l=2$, then by Lemma \ref{linearstand-binomial}, $k=\pd_S(S/J_G)$. It follows from Theorem \ref{cone-depth} 
	that $k=\pd_{S_H}(S_H/J_H)+2\geq n+1$. Therefore, by  Theorem \ref{Betti-cone},
	\[\beta_{k,k+2}^S\left(\frac{S}{J_G}\right)=\beta_{k-2,k}^{S_H}\left(\frac{S_H}{J_{H}}\right)+(k-1)\binom{n+1}{k+1}-(k-1)(\sk_{k}(H)+\sk_{k+1}(H))=\beta_{k-2,k}^{S_H}\left(\frac{S_H}{J_H}\right).\] Hence, the assertion follows.
\end{proof}
Let $H$ be a connected graph on $[n]$. Also, let $G=K_q*H$. Then by using  Proposition  \ref{extremal-cone}, we conclude that $S_H/J_H$ admits  unique extremal
Betti number if and only if $S/J_G$ admits unique extremal Betti number. In particular, if $\beta_{p,p+r}^{S_H}(S_H/J_H)$ is an extremal Betti number, then 
$\beta_{p+2q,p+2q+r}^S(S/J_G)$ is an extremal and $\beta_{p,p+r}^{S_H}(S_H/J_H)=\beta_{p+2q,p+2q+r}^S(S/J_G)$.

\section{Depth of join of graphs}
In this section, we compute the depth of binomial edge ideal of join of two graphs. Let $G_1$ and $G_2$ be  graphs on $[n_1]$ and $[n_2]$, respectively
with $n_1,n_2 \geq 2$. We assume that both $G_1$ and $G_2$ are not complete. The join of $G_1$ and $G_2$, denoted by $G_1*G_2$  is 
the graph with vertex set $[n_1] \sqcup [n_2]$ and the edge set
$E(G_1*G_2)= E(G_1) \cup E(G_2) \cup \{\{i,j\} : i \in [n_1], j \in [n_2]\}$. Let $G=G_1* G_2$. It follows from 
\cite[Propositions 4.1, 4.5, 4.14]{KM6} that if $P_T(G)$ is a minimal prime of $J_G$ for some $T\subseteq [n_1]\sqcup [n_2]$, then either $T=\emptyset$ or
$[n_1]\subseteq T$ or $[n_2]\subseteq T$. Therefore, by virtue of \cite[Theorem 3.2, Corollary 3.9]{HH1}, we have $$J_G=P_{\emptyset}(G)\cap \big((x_i,y_i:i\in [n_1])+J_{G_2}\big)\cap \big((x_j,y_j:j\in [n_2])+J_{G_1}\big).$$
Set $Q_1=(x_i,y_i:i\in [n_2])+J_{G_1}$, $Q_2=(x_i,y_i:i\in [n_1])+J_{G_2}$ and $Q_3=P_{\emptyset}(G)\cap Q_2$. One can see that 
$Q_2+P_{\emptyset}(G)=(x_i,y_i:i\in [n_1])+J_{K_{n_2}}$ and $Q_1+Q_3=(x_i,y_i:i\in [n_2])+J_{K_{n_1}}.$ Thus, we have the following short exact sequences:
\begin{align}\label{1stseq}
	0\longrightarrow \frac{S}{Q_{3}} \longrightarrow \frac{S}{P_{\emptyset}(G)} \oplus \frac{S}{Q_{2}} \longrightarrow \frac{S}{(x_i,y_i:i\in [n_1])+J_{K_{n_2}}} \longrightarrow 0
\end{align}
and 
\begin{align}\label{2ndseq}
	0\longrightarrow \frac{S}{J_{G}}\longrightarrow \frac{S}{Q_{1}}\oplus \frac{S}{Q_{3}}\longrightarrow \frac{S}{(x_i,y_i:i\in [n_2])+J_{K_{n_1}}} \longrightarrow 0.
\end{align}
Let $S_i = \mathbb{K}[x_j,y_j : j \in [n_i]]$ for $i=1,2$. Observe that  $\depth_S(S/Q_1)=\depth _{S_1}(S_1/J_{G_1})$, $\depth_S(S/(Q_2+P_\emptyset(G)))=n_2+1$, $\depth_S(S/(Q_1+Q_3))=n_1+1$ and 
$\depth_S(S/Q_2)=\depth _{S_2}(S_2/J_{G_2})$. Thus, using Lemma \ref{depth-lemma} in short  exact sequence \eqref{1stseq}, 
\begin{align}\label{depth-Q_3}
	\depth_S(S/Q_3)\geq \min \{\depth_{S_2}(S_2/J_{G_2}),n_2+2\}
\end{align} and hence from 
the exact sequence \eqref{2ndseq} that
\begin{align}\label{depthresult}
	\depth_S(S/J_G)\geq \min \{\depth_{S_1}(S_1/J_{G_1}),\depth_{S_2}(S_2/J_{G_2}),n_1+2,n_2+2\}.
\end{align}
First, we give exact formula for depth of binomial edge ideal of join of two connected graphs.
\begin{theorem}\label{depththm2}
	Let $G=G_1* G_2$ be join of $G_1$ and $G_2$, where $G_1$ and $G_2$ be two connected graphs on vertex sets $[n_1]$ and $[n_2]$ respectively.
	Then \[\depth_S(S/J_G)=\displaystyle{\min_{i ={1,2}}\depth_{S_i}(S_i/J_{G_i})}.\]
\end{theorem}
\begin{proof}
	First we prove that  $\depth_{S}(S/Q_3)=\depth_{S_2}(S_2/J_{G_2})$. If $\depth_{S_2}(S_2/J_{G_2})<n_2+1$, it follows from short exact sequence \eqref{1stseq} and Lemma \ref{depth-lemma} 
	that $\depth_{S}(S/Q_3)=\depth_{S_2}(S_2/J_{G_2})$. Now, we assume that $\depth_{S_2}(S_2/J_{G_2})=n_2+1$, by Auslander-Buchsbaum formula, 
	$\pd_{S_2}(S_2/J_{G_2})=n_2-1$. Since, $G_2$ is not a complete graph, by Lemma \ref{linearstand-binomial}, there exists $j \geq 2$ such that 
	$\beta_{n_2-1,n_2-1+j}^{S_2}(S_2/J_{G_2}) \neq 0$ which  implies that $\beta_{p,p+j}^{S}(S/Q_2) \neq 0$, where $p=2n_1+n_2-1$. Note that $\pd_S(S/P_{\emptyset}(G))=n_1+n_2-1$ and
	$\pd_S(S/Q_2)=2n_1+n_2-1=\pd_S(S/((x_i,y_i: i \in [n_1])+J_{K_{n_2}}))$. Now consider the long exact 
	sequence of Tor  in homological degree $p$ corresponding to the short exact sequence \eqref{1stseq}
	\begin{align}\label{Tor-depth}
		0 \rightarrow \Tor_{p,p+j}^{S}\left( \frac{S}{Q_3},\KK\right)\rightarrow \Tor_{p,p+j}^{S}\left( \frac{S}{Q_2},\KK\right) \rightarrow  \Tor_{p,p+j}^{S}\left(\frac{S}{P_{\emptyset}(G)+Q_2},\KK\right) \rightarrow  \cdots
	\end{align}
	Since $j\geq 2$, $\beta_{p,p+j}^S(S/(P_{\emptyset}(G)+Q_2))= 0$, which further implies that $\beta_{p,p+j}^S(S/Q_3)\neq 0$. Thus,
	$\pd_S(S/Q_3)\geq p$ and hence $\depth_S(S/Q_3)\leq 2n_1+2n_2-p=n_2+1$. Hence, by \eqref{depth-Q_3},
	$\depth_S(S/Q_3)=n_2+1=\depth_{S_2}(S_2/J_{G_2})$.  
	
	Now, if $\min\{\depth_{S_i}(S_i/J_{G_i}): i=1,2\}<n_1+1$, then using Lemma \ref{depth-lemma} in short exact sequence \eqref{2ndseq}, we get the desired 
	result. Otherwise, by Remark \ref{pd-lower}, we have $\depth_{S_1}(S_1/J_{G_1})=n_1+1=\min\{\depth_{S_i}(S_i/J_{G_i}): i=1,2\}$ and therefore
	$\pd_{S_1}(S_1/J_{G_1})=n_1-1$.
	Since, $G_1$ is not a complete graph, by Lemma \ref{linearstand-binomial}, there exists $l \geq 2$ such that $\beta_{n_1-1,n_1-1+l}^{S_1}(S_1/J_{G_1}) \neq 0$
	which further implies that $\beta_{n_1+2n_2-1,n_1+2n_2-1+l}^{S}(S/Q_1) \neq 0$. Note that, $\pd_S(S/Q_1)=n_1+2n_2-1=\pd_S(S/((x_i,y_i: i \in [n_2])+J_{K_{n_1}}))$.
	The long exact sequence of Tor in homological degree $q=n_1+2n_2-1$ and graded degree $q+l$ corresponding to \eqref{2ndseq} is
	\begin{align}\label{Tor-depth2}
		0 \rightarrow \Tor_{q,q+l}^{S}\left( \frac{S}{J_G},\KK\right)\rightarrow \Tor_{q,q+l}^{S}\left( \frac{S}{Q_1},\KK\right)\oplus  \Tor_{q,q+l}^{S}\left( \frac{S}{Q_3},\KK\right)\rightarrow  0.
	\end{align}
	Since,  $\beta_{q,q+l}^{S}(S/Q_1) \neq 0$, we have $\beta_{q,q+l}^S(S/J_G) \neq 0$. Therefore, $\pd_S(S/J_G)\geq q$ and hence $\depth_S(S/J_G)\leq 2n_1+2n_2-q=n_1+1$.
	It follows from \eqref{depthresult} that $\depth_S(S/J_G)\geq n_1+1$. Hence, the desired result follows.	
\end{proof}
We now illustrate our result by the following example. A \textit{block} of a graph is a  maximal nontrivial connected  subgraph with no cut vertex. A connected graph is said to be a \textit{block graph} if every block of that graph is a complete graph.

\begin{example}
	If $G_1$ be a connected block graph and $G_2=C_{n_2}$ with $n_1\geq n_2 \geq 4$, then by virtue of \cite[Theorem 1.1]{her1} $\depth_{S_1}(S_1/J_{G_1})=n_1+1$. 
	By   \cite[Corollary 16]{SZ} that $\depth_{S_2}(S_2/J_{G_2})=n_2$. Hence,  $\depth_S(S/J_{G_1* G_2})=\depth_{S_2}(S_2/J_{G_2})=n_2$.
\end{example} 

Now, we move on to study the join of a connected graph and a disconnected graph.
\begin{theorem}\label{depththm1}
	Let $G_1$ be a connected graph on the vertex set $[n_1]$ and $G_2$ be a disconnected graph on the vertex set $[n_2]$. Then  
	\[\depth_S(S/J_G)=\min\{\depth_{S_1}(S_1/J_{G_1}),\depth_{S_2}(S_2/J_{G_2}),n_2+2\}.\]
\end{theorem}
\begin{proof}
	We claim  that  $\depth_{S}(S/Q_3)=\min\{\depth_{S_2}(S_2/J_{G_2}),n_2+2\}$. First assume that $n_2+1 <\depth_{S_2}(S_2/J_{G_2})$. Therefore, the 
	claim follows from the short exact sequence \eqref{1stseq} and Lemma \ref{depth-lemma}. If $n_2+1>\depth_{S_2}(S_2/J_{G_2})$, it follows from short exact sequence
	\eqref{1stseq} that $\depth_{S}(S/Q_3)=\depth_{S_2}(S_2/J_{G_2})$. Now, we assume that $\depth_{S_2}(S_2/J_{G_2})=n_2+1$ which implies that 
	$\pd_{S_2}(S_2/J_{G_2})=n_2-1$. Since, $G_2$ is a disconnected graph, by Remark \ref{linearstand-binomial3}, there exists $j \geq 2$ such that
	$\beta_{n_2-1,n_2-1+j}^{S_2}(S_2/J_{G_2}) \neq 0$ which further implies that $\beta_{p,p+j}^{S}(S/Q_2) \neq 0$, where $p=2n_1+n_2-1$. Note that 
	$\pd_S(S/P_{\emptyset}(G))=n_1+n_2-1$, $\pd_S(S/Q_2)=2n_1+n_2-1=\pd_S(S/((x_i,y_i: i \in [n_1])+J_{K_{n_2}}))$. Now, consider the long exact sequence of Tor \eqref{Tor-depth}.
	Since $\beta_{p,p+j}^{S}(S/Q_2)\neq 0$, we have that $\beta^S_{p,p+j}(S/Q_3)\neq 0$. Thus, $\pd_S(S/Q_3)\geq p$ and hence $\depth_S(S/Q_3)\leq 2n_1+2n_2-p=n_2+1$. 
	Therefore, by \eqref{depth-Q_3}, $\depth_S(S/Q_3)=n_2+1=\depth_{S_2}(S_2/J_{G_2})$. Hence, we have	\[\depth_{S}(S/Q_3)=\min\{\depth_{S_2}(S_2/J_{G_2}),n_2+2\}.\]
	
	Now, if $\min\{\depth_{S_1}(S_1/J_{G_1}),\depth_{S_2}(S_2/J_{G_2}),n_2+2\}<n_1+1$, then by applying Lemma \ref{depth-lemma} in short exact sequence \eqref{2ndseq}, 
	we get the desired result. Otherwise, we have $\min\{\depth_{S_1}(S_1/J_{G_1}),\depth_{S_2}(S_2/J_{G_2}),n_2+2\}\geq n_1+1$ and hence by Remark \ref{pd-lower}, 
	$\depth_{S_1}(S_1/J_{G_1})=n_1+1$. Therefore $\pd_{S_1}(S_1/J_{G_1})=n_1-1$.
	Since, $G_1$ is not a complete graph, by Lemma \ref{linearstand-binomial}, there exists $l \geq 2$ such that $\beta_{n_1-1,n_1-1+l}^{S_1}(S_1/J_{G_1}) \neq 0$ which 
	implies that $\beta_{q,q+l}^{S}(S/Q_1) \neq 0$, where $q=n_1+2n_2-1$. Note  that $\pd_S(S/Q_1)=q=\pd_S(S/((x_i,y_i: i \in [n_2])+J_{K_{n_1}}))$.
	Consider, the long exact sequence of Tor \eqref{Tor-depth2}.
	Since,  $\beta_{q,q+l}^{S}(S/Q_1) \neq 0$, we have $\beta_{q,q+l}^S(S/J_G) \neq 0$. Therefore, $\pd_S(S/J_G)\geq q$ and hence $\depth_S(S/J_G)\leq 2n_1+2n_2-q=n_1+1$. 
	It follows from \eqref{depthresult} that $\depth_S(S/J_G)\geq n_1+1$. Hence, the assertion follows.
\end{proof}	
We now compute the depth of the binomial edge ideal of join of two disconnected  graphs. 
\begin{theorem}\label{depththm3}
	Let $G=G_1*G_2$ be the join of $G_1$ and $G_2$, where $G_1$ and $G_2$ are disconnected graphs on $[n_1]$ and $[n_2]$ respectively. Assume that $n_2 \geq n_1$. Then
	\[\depth_S(S/J_G)=\min\{\depth_{S_1}(S_1/J_{G_1}),\depth_{S_2}(S_2/J_{G_2}),n_1+2\}.\]  
\end{theorem}
\begin{proof}
	It follows from the proof of Theorem \ref{depththm1} that \[\depth_{S}(S/Q_3)=\min\{\depth_{S_2}(S_2/J_{G_2}),n_2+2\}.\]
	
	Now, if $\min\{\depth_{S_1}(S_1/J_{G_1}),\depth_{S_2}(S_2/J_{G_2}),n_2+2\}<n_1+1$, then using Lemma \ref{depth-lemma} in short exact sequence \eqref{2ndseq}, we get the desired result.
	
	If $\min\{\depth_{S_1}(S_1/J_{G_1}),\depth_{S_2}(S_2/J_{G_2}),n_2+2\}=n_1+1$, then either $\depth_{S_1}(S_1/J_{G_1})=n_1+1$ or $\depth_S(S/Q_3)=n_1+1$. Now, if $\depth_{S_1}(S_1/J_{G_1})=n_1+1$,
	then by Auslander-Buchsbaum formula, $\pd_{S_1}(S_1/J_{G_1})=n_1-1$. Therefore, by virtue of Remark \ref{linearstand-binomial3} there exists $j \geq 2$ such that
	$\beta_{n_1-1,n_1-1+j}^{S_1}(S_1/J_{G_1})\neq 0$  which  implies that $\beta_{q,q+j}^{S}(S/Q_1) \neq 0$, where $q=n_1+2n_2-1$. Note  that $\pd_S(S/Q_1)=q=\pd_S(S/((x_i,y_i: i \in [n_2])+J_{K_{n_1}}))$.
	Consider, the long exact sequence of Tor \eqref{Tor-depth2} in graded degree $q+j$.
	Since,  $\beta_{q,q+j}^{S}(S/Q_1) \neq 0$, we have $\beta_{q,q+j}^S(S/J_G) \neq 0$. Therefore, $\pd_S(S/J_G)\geq q$ and hence $\depth_S(S/J_G)\leq 2n_1+2n_2-q=n_1+1$. Now, the assertion
	follows from \eqref{depthresult}.  Assume  now that  $\depth_S(S/Q_3)=n_1+1$. Since, $n_1\leq n_2$, $\depth_S(S/Q_3)=n_1+1=\depth_{S_2}(S_2/J_{G_2})=\depth_S(S/Q_2)$.  Note that $\pd_S(S/Q_3) =q =\pd_S(S/Q_2)$.
	Since, $G_2$ is a disconnected graph and $\depth_{S_2}(S_2/J_{G_2})=n_1+1 \leq n_2+1$, either $G_2$ has atleast two nontrivial components or $G_2$ has one nontrivial component which is not complete. In first case, by Lemma \ref{linearstand-binomial2}, there exists $j \geq 2$ such that $\beta_{2n_2-n_1-1,2n_2-n_1-1+j}^{S_2}(S_2/J_{G_2}) \neq 0$ which
	further implies that $\beta_{q,q+j}^{S}(S/Q_2) \neq 0$. If $G_2$ has exactly one nontrivial component say $H$, then $\pd_{S_2}(S_2/J_{G_2})=\pd_{S_H}(S_H/J_{H})=2n_2-n_1-1$, where $S_H=\KK[x_j,y_j : j \in V(H)]$. Now, by Lemma \ref{linearstand-binomial}, there exists $j \geq 2$ such that $\beta_{2n_2-n_1-1,2n_2-n_1-1+j}^{S_2}(S_2/J_{G_2}) \neq 0$ which
	further implies that $\beta_{q,q+j}^{S}(S/Q_2) \neq 0$. The long exact sequence of Tor corresponding to \eqref{1stseq} in homological degree $q$ and graded degree $q+j$ is
	\begin{align*}
		\cdots \rightarrow \Tor_{q,q+j}^{S}\left( \frac{S}{Q_3},\KK\right)\rightarrow \Tor_{q,q+j}^{S}\left( \frac{S}{Q_2},\KK\right) \rightarrow  0.
	\end{align*}
	Therefore, $\beta_{q,q+j}^S(S/Q_3)\neq 0$. Thus, it follows from  \eqref{Tor-depth2} that $\beta_{q,q+j}^S(S/J_G)\neq 0$. Therefore, $\pd_S(S/J_G)\geq q$ and hence,
	$\depth_S(S/J_G) \leq 2n_1+2n_2-q=n_1+1$. Now, along with \eqref{depthresult}, we get the assertion.
	
	Also, if $\min\{\depth_{S_1}(S_1/J_{G_1}),\depth_{S_2}(S_2/J_{G_2}),n_2+2\}> n_1+1$, then again using Lemma \ref{depth-lemma} in the short exact sequence \eqref{2ndseq}, $\depth_S(S/J_G)=n_1+2$.  Hence, the desired result follows.
\end{proof}	
As an immediate consequence, we obtain the depth of complete multipartite graph.
\begin{corollary}\label{com-bipartite}
	Let $G=K_{n_1,\cdots,n_k}$ be a complete multipartite graph with $2 \leq n_1\leq \cdots \leq n_k$. Then $\depth_S(S/J_G)= n_1+2.$
\end{corollary}

\section{Construction of Graph}
In this section, we construct a graph $G$ such that $\reg(S/J_G)=r$ and the number of extremal Betti numbers  of $S/J_G$ is $b$,
where $1 \leq b<r$. We now set some notation for the rest of this section. Let $G_1$ and $G_2$ be two connected  graphs which are 
not complete on the vertex sets $[n_1]$ and $[n_2]$, respectively. Let $p_i=\pd_{S_i}(S_i/J_{G_i})$ and $r_i=\reg(S_i/J_{G_i})$ for $i=1,2$.
By Remark \ref{pd-lower}, $p_i \geq n_i-1$, for $i=1,2$.

\begin{lemma}
	Let $G_1$ and $G_2$ be graphs on $[n_1]$ and $[n_2]$, respectively. Let $G=G_1*G_2$. If $\reg(S/J_G)=2$, then $S/J_G$ admits unique extremal Betti number.
\end{lemma}
\begin{proof}
	Proof follows from the Lemma \ref{linearstand-binomial}.
\end{proof}
We consider the long exact sequence of Tor corresponding to the exact sequence \eqref{1stseq}

\begin{multline}\label{Tor-1st}
	\cdots \rightarrow \Tor_{k,k+l}^{S}\left( \frac{S}{Q_3},\KK\right)\rightarrow \Tor_{k,k+l}^{S}\left( \frac{S}{P_{\emptyset}(G)},\KK\right)
	\oplus \Tor_{k,k+l}^{S}\left(\frac{S}{Q_2},\KK\right)\\ \rightarrow \Tor_{k,k+l}^{S}\left(\frac{S}{(x_i,y_i:i\in [n_1])+J_{K_{n_2}}},\KK\right) \rightarrow \Tor_{k-1,k+l}^{S}\left( \frac{S}{Q_3},\KK\right)\rightarrow \cdots
\end{multline}
where $Q_2=(x_i,y_i:i\in [n_1])+J_{G_2}$ and $Q_3=P_{\emptyset}(G)\cap Q_2$. 

Now, we prove that  extremal Betti numbers of $S/Q_3$ and $S/Q_2$ coincide in terms of position and value.

\begin{lemma}\label{tech-lemma2}
	Let $G=G_1*G_2$ be the join graph on $[n_1]\sqcup [n_2]$. If $\beta_{k,k+l}^{S_2}(\frac{S_2}{J_{G_2}})$ is an extremal 
	Betti number, then $\beta_{k+2n_1,k+2n_1+l}^S(\frac{S}{Q_3})$ is an extremal Betti number. Moreover, extremal Betti number
	of $S/Q_3$ is of the form $\beta_{k+2n_1,k+2n_1+l}^S(\frac{S}{Q_3})$, where $\beta_{k,k+l}^{S_2}(\frac{S_2}{J_{G_2}})$ is an extremal Betti number.
\end{lemma}
\begin{proof}
	It follows from proof of Theorem \ref{depththm2} that $\depth_S(S/Q_3)=\depth_{S_2}(S_2/J_{G_2})$. Thus, $\pd_S(S/Q_3)=p_2 +2n_1$.  Since, $\pd_S(S/P_{\emptyset}(G))=n_1+n_2-1$
	and $\pd_S(S/((x_i,y_i: i \in [n_1])+J_{K_{n_2}}))=2n_1+n_2-1$, by Lemma \ref{linearstand-binomial},
	\[ \Tor_{p_2+2n_1,p_2+2n_1+1}^{S}\left(\frac{S}{Q_2},\KK\right) \simeq \Tor_{p_2,p_2+1}^{S_2}\left( \frac{S_2}{J_{G_2}},\KK\right)=0.\]	
	Also, $\Tor_{p_2+2n_1+1,p_2+2n_1+1}^S(S/((x_i,y_i:i\in [n_1])+J_{K_{n_2}})) =0$. By considering
	the long exact sequence of Tor	\eqref{Tor-1st} in homological degree $p_2 +2n_1$ and graded degree $p_2 +2n_1+1$, we get, $\beta_{p_2+2n_1,p_2+2n_1+1}^S(\frac{S}{Q_3})=0$. Now, let 
	$\beta_{k,k+l}^{S_2}(\frac{S_2}{J_{G_2}})$ is an extremal Betti number. By virtue of Lemma \ref{linearstand-binomial}, $l \geq 2$ and therefore the long exact sequence of 
	Tor \eqref{Tor-1st} in homological degree $p=k+2n_1$ and graded degree $p+l$ is
	\begin{eqnarray*}
		\rightarrow	\Tor_{p+1,p+l}^{S}\left(\frac{S}{(x_i,y_i:i\in [n_1])+J_{K_{n_2}}},\KK\right)\rightarrow \Tor_{p,p+l}^{S}\left( \frac{S}{Q_3},\KK\right)\rightarrow   \Tor_{p,p+l}^{S}\left(\frac{S}{Q_2},\KK\right) \rightarrow 0.
	\end{eqnarray*}
	Now if $l>2$, then we have
	\[\Tor_{p,p+l}^{S}\left( \frac{S}{Q_3},\KK\right)\simeq \Tor_{p,p+l}^{S}\left(\frac{S}{Q_2},\KK\right) \simeq \Tor_{k,k+l}^{S_2}\left( \frac{S_2}{J_{G_2}},\KK\right).\]
	Thus, $\beta_{p,p+l}^S(\frac{S}{Q_3})\neq 0$. Let $(s,t) \neq (p,l)$ with $s \geq p,t \geq l$. Taking homological degree $s\geq p$ and graded degree $s+t \geq p+l$ in \eqref{Tor-1st}, we have
	\[\Tor_{s,s+t}^{S}\left( \frac{S}{Q_3},\KK\right)\simeq \Tor_{s,s+t}^{S}\left(\frac{S}{Q_2},\KK\right) \simeq \Tor_{s-2n_1,s+t-2n_1}^{S_2}\left( \frac{S_2}{J_{G_2}},\KK\right).\] Note that
	$s-2n_1 \geq k, t\geq l$ and $(s-2n_1,t)\neq (k,l)$. Therefore, $$\beta_{s,s+t}^S\left(\frac{S}{Q_3}\right)=\beta_{s-2n_1,s-2n_1+t}^{S_2}\left(\frac{S_2}{J_{G_2}}\right) =0.$$ Hence, 
	$\beta_{p,p+l}^S(\frac{S}{Q_3})$  is an extremal Betti number.
	Now we assume that $l= 2$.  It follows from Lemma \ref{linearstand-binomial} that $k=p_2$.   Note that
	$\pd_S(S/((x_i,y_i: i \in [n_1])+J_{K_{n_2}}))=2n_1+n_2-1 \leq p_2+2n_1$. Now, consider the long exact sequence of Tor \eqref{Tor-1st} in homological degree
	$p=p_2+2n_1$, $\Tor_{p_2+2n_1,p_2+2n_1+j}^{S}\left( \frac{S}{Q_3},\KK\right) \simeq \Tor_{p_2,p_2+j}^{S_2}\left( \frac{S_2}{J_{G_2}},\KK\right),$ for $j \geq 2$. 
	Since,  $\beta_{p_2,p_2+2}^{S_2}(\frac{S_2}{J_{G_2}})$ is an extremal Betti number, $\beta_{p,p+2}^S(\frac{S}{Q_3})$ is an extremal Betti number.
	
	Now, let $\beta_{i,i+j}^S(\frac{S}{Q_3})$ be an extremal Betti number. Consider, the long exact sequence of Tor \eqref{Tor-1st} in homological degree $i$ and graded degree $i+j$.
	If $j>2$, then \[\Tor_{i,i+j}^{S}\left( \frac{S}{Q_3},\KK\right)\simeq \Tor_{i,i+j}^{S}\left(\frac{S}{Q_2},\KK\right) \simeq \Tor_{i-2n_1,i-2n_1+j}^{S_2}\left( \frac{S_2}{J_{G_2}},\KK\right).\]
	Therefore, $\beta_{i-2n_1,i-2n_1+j}^{S_2}( \frac{S_2}{J_{G_2}})\neq 0$. Now, if for some $s \geq i-2n_1, t\geq j$ with $(s,t) \neq (i-2n_1,j)$, $\beta_{s,s+t}^{S_2}(\frac{S_2}{J_{G_2}}) \neq 0$. 
	Then, $\beta_{s+2n_1,s+2n_1+t}^S(\frac{S}{Q_3}) \neq 0$ and $s+2n_1 \geq i, t\geq j$ and $(s+2n_1,t)\neq (i,j)$ which is a contradiction. Hence, $\beta_{i-2n_1,i-2n_1+j}^{S_2}(\frac{S_2}{J_{G_2}})$  
	is an extremal Betti number. If $j=2$, then by Lemma \ref{linearstand-binomial}, $i=p_2+2n_1$. For any $l \geq 2$, we have $\Tor_{p_2+2n_1,p_2+2n_1+l}^{S}\left( \frac{S}{Q_3},\KK\right) \simeq \Tor_{p_2,p_2+l}^{S_2}\left( \frac{S_2}{J_{G_2}},\KK\right)$. 
	Therefore, $\beta_{p_2,p_2+2}^{S_2}( \frac{S_2}{J_{G_2}})$ is an extremal Betti number. This completes the proof.
\end{proof}

It follows from above theorem that $\reg(S/Q_3)=\reg(S/Q_2)=\reg(S_2/J_{G_2})$. 
Assume that $2\leq r_1\leq r_2$ and $p_1\leq p_2$. Since, $G_1$ is a connected graph with $r_1 \geq 2$, by Lemma \ref{linearstand-binomial}, 
if $\beta_{k,k+l}^{S_1}(\frac{S_1}{J_{G_1}})$ is an extremal Betti number, then $l \geq 2$.
We now consider long exact sequence of Tor in homological degree $k$ and graded degree $k+l \geq k+2$ corresponding to the exact sequence \eqref{2ndseq}, 
\begin{align}\label{Tor-2nd}
	\cdots \rightarrow  \Tor_{k,k+l}^S\left(\frac{S}{J_G},\KK \right)\rightarrow \Tor_{k,k+l}^{S}\left( \frac{S}{Q_3},\KK\right) \oplus \Tor_{k,k+l}^{S}\left( \frac{S}{Q_1},\KK\right) \rightarrow  0,
\end{align}
where $Q_1=(x_i,y_i:i\in [n_2])+J_{G_1}$.
\begin{lemma}\label{tech-lemma}
	Let $G=G_1*G_2$ be the join graph on $[n_1]\sqcup [n_2]$.	Suppose $\depth_{S_1}(S_1/J_{G_1})<\depth_{S_2}(S_2/J_{G_2})$ i.e., $p_2+2n_1<p_1+2n_2$. If $\beta_{k,k+l}^S(S/J_G)$ is an extremal
	Betti number, then $$\Tor_{k,k+l}^S\left(\frac{S}{J_G}\right)\simeq \Tor_{k,k+l}^{S}\left(\frac{S}{Q_3}\right)\oplus \Tor_{k,k+l}^{S}\left(
	\frac{S}{Q_1}\right).$$
\end{lemma}
\begin{proof}
	Note that $\pd_S(S/((x_i,y_i:i\in [n_2])+J_{K_{n_1}}))=2n_2+n_1-1< 2n_2+p_1+1,$ by Remark \ref{pd-lower}. Therefore,
	$\Tor_{p_1+2n_2+1,p_1+2n_2+1+j}^S\left(\frac{S}{(x_i,y_i:i\in [n_2])+J_{K_{n_1}}}, \KK\right)=0$, for $j \geq 1$. If $l=2$, by Lemma \ref{linearstand-binomial}, $k=\pd_S(S/J_G)=p_1+2n_2$. Also if $l>2$, then 
	$$\Tor_{k,k+l}^S\left(\frac{S}{(x_i,y_i:i\in [n_2])+J_{K_{n_1}}}, \KK\right)=\Tor_{k+1,k+l}^S\left(\frac{S}{(x_i,y_i:i\in [n_2])+J_{K_{n_1}}}, \KK\right)=0.$$
	Hence, by \eqref{Tor-2nd} $$\Tor_{k,k+l}^S\left(\frac{S}{J_G}\right)\simeq \Tor_{k,k+l}^S\left(\frac{S}{Q_3}\right)\oplus \Tor_{k,k+l}^S\left(
	\frac{S}{Q_1}\right).$$
\end{proof}

By \cite[Theorem 2.1]{KM1}, $\reg(S/J_{K_n}) =1$ and hence $S/J_{K_n}$ admits unique extremal Betti number. So for $r=b=1$, consider $G=K_n$. 
Now assume that $r\geq 2$. It follows from Betti diagram and Lemmas \ref{linearstand-binomial}, \ref{linearstand-binomial2} that $b\leq r-1$. 
\begin{theorem}\label{exist}
	Let $r$ and $b$ be two positive integers with $1\leq b\leq r-1$. Then there exists a graph $G=G_{r,b}$ such that $\reg(S/J_G)=r$ and 
	the number of extremal Betti numbers of $S/J_{G}$ is $b$. 
\end{theorem}
\begin{proof}
	Take $G=G_{r,b}=P_{r-b+2}*\cdots *P_{r+1}$. Note that $n=|V(G)|=br-\frac{b(b-3)}{2}$ and by recursively applying Theorem \ref{depththm2}, $\depth_S(S/J_G)=r-b+3$. Now, by Auslander-Buchsbaum 
	formula, $p =\pd_S(S/J_G)=(2b-1)r-(b-1)(b-3)$. It follows from \cite[Theorem 2.1]{KMJA} that $\reg(S/J_{G})=r.$ We now prove that extremal Betti numbers of $S/J_G$ are precisely 
	$\beta^{S}_{p-i,p+r-b+1} \left(\frac{S}{J_{G}}\right)=1$, for $0\leq i\leq b-1$. 
	We proceed by induction on  $b$. If $b=1$, then $G=G_{r,1}=P_{r+1}$. By \cite[Corollary 1.2]{her1}, $J_{G}$ is complete intersection ideal, $S/J_{G}$ has a unique 
	extremal Betti number, $\beta_{r,2r}^S \left( \frac{S}{J_{G}}\right)=1$. Now, assume that $b>1$ and extremal Betti numbers of $S_2/J_{G_2}$ are precisely  
	$\beta^{S_2}_{p_2-i,p_2+r-b+2} \left(\frac{S_2}{J_{G_{2}}}\right)$ for $0\leq i\leq b-2$, where $G_2=G_{r,b-1}=P_{r-b+3}*\cdots *P_{r+1}$, $S_2=\mathbb{K}[x_j,y_j : j \in V(G_2)]$ and $p_2=\pd_{S_2}(S_2/J_{G_2})$. 
	Observe that $n_2 =|V(G_2)|=(b-1)r-\frac{(b-1)(b-4)}{2}$. Also, by Theorem \ref{depththm2}, $\depth_{S_2}(S_2/J_{G_2})=r-b+4$. Thus, by Auslander-Buchsbaum formula, $p_2=2n_2-r+b-4=(2b-3)r-(b-2)(b-4).$ 
	Set $G_1=P_{r-b+2}$, $n_1=|V(G_1)|=r-b+2$ and $S_1=\mathbb{K}[x_j,y_j : j \in V(G_1)]$. By \cite[Corollary 1.2]{her1}, $p_1=\pd_{S_{1}} ({S_{1}}/{J_{G_1}})=r-b+1=r_1=\reg(S_1/J_{G_1})$. Note that $G=G_{r,b}=G_1*G_2$.
	Since $J_{G_1}$ is complete intersection ideal, $\beta^{S_{1}}_{p_1,p_1+r-b+1}\left( \frac{S_{1}}{J_{G_1}}\right)=1$ is the extremal Betti number of $S_{1}/J_{G_1}$. 
	Therefore, $\beta^S_{p_1+2n_2,p_1+2n_2+r-b+1}\left(\frac{S}{Q_1}\right)$ is unique extremal Betti number of $S/Q_1$. Note that $p_1+2n_2=(2b-1)r-(b-1)(b-3)=p$.
	
	Let $\beta_{k,k+l}^S(S/J_G)$ is an extremal Betti number. Now, by Lemma \ref{tech-lemma}, we have 
	$$\Tor_{k,k+l}^S\left(\frac{S}{J_G}\right)\simeq \Tor_{k,k+l}^{S}\left(\frac{S}{Q_3}\right)\oplus \Tor_{k,k+l}^{S}\left(
	\frac{S}{Q_1}\right).$$
	It follows from Lemma \ref{tech-lemma2} that  extremal Betti numbers of $S/Q_3$ are   $\beta^S_{p_2+2n_1-i,p_2+2n_1+r-b+2}\left(\frac{S}{Q_3}\right)$,
	for $0\leq i\leq b-2$. Note that $p_2+2n_1=p_2+2(r-b+2)=p-1$. Therefore,  extremal Betti numbers of $S/Q_3$ are 
	$\beta^S_{p-i-1,p+r-b+1}\left(\frac{S}{Q_3}\right)=1$, for $0\leq i \leq b-2$.  Since $b\leq r-1$, $r-b+1\geq 2$. So, for $j\geq r-b+2\geq 3$ and $1 \leq k \leq p$, 
	$$\Tor_{k,k+j}^S\left(\frac{S}{J_G}\right)\simeq \Tor_{k,k+j}^S\left(\frac{S}{Q_3}\right).$$
	Therefore, $\beta_{p-i-1,p+r-b+1}^S\left(\frac{S}{J_G}\right) =1$, for $0 \leq i \leq b-2$ are  extremal Betti numbers of $S/J_G$. Now, it remains to prove that $\beta_{p,p+r-b+1}^S\left(\frac{S}{J_G}\right)$ is an extremal Betti number.

	Taking long exact sequence of Tor \eqref{Tor-1st} in homological degree $p$ and graded degree $p +j=p +r-b+1\geq p +2$, we have 
	$\Tor^S_{p,p+r-b+1}\left(\frac{S}{J_G}\right)\simeq \Tor^S_{p,p+r-b+1}\left(\frac{S}{Q_1}\right)\simeq \Tor^{S_1}_{p_1,p_1+r-b+1}\left(\frac{S_1}{J_{G_1}}\right)$.
	Therefore, $\beta^S_{p,p+r-b+1}\left(\frac{S}{J_G}\right)=1$ is also an extremal Betti number. Hence the number of extremal Betti numbers of $S/J_G$ is $b$ and the extremal Betti
	number are of the form $\beta^{S}_{p-i,p+r-b+1} \left(\frac{S}{J_{G}}\right)=1$, for $0\leq i\leq b-1$.
\end{proof}
Observe that the projective dimension of $S/J_{G_{r,b}}$ is very large. Hence the following  question arises:
\begin{question}
	Does there exist a graph $G$ such that the projective dimension is bounded by a linear function of $b$ and $r$, where $r=\reg(S/J_G)$ and $b$ is the number of extremal Betti numbers of $S/J_G$?
\end{question}
\vskip 2mm
\noindent
\textbf{Acknowledgement:} The authors are grateful to their advisor A. V. Jayanthan for
constant support, valuable ideas and suggestions. The first author thanks the National Board
for Higher Mathematics, India for the financial support. The second author thanks 
University Grants Commission, Government of India for the financial support.
\bibliographystyle{plain}
\bibliography{biblog}
\end{document}